\documentclass{amsart}
\pdfoutput=1
\usepackage[colorlinks]{hyperref}
\usepackage{booktabs}
\usepackage[all]{xy}

\newtheorem*{lemma*}{Lemma}

\DeclareMathOperator{\Aut}{Aut}
\DeclareMathOperator{\PGL}{PGL}
\DeclareMathOperator{\Otilde}{\tilde{O}}

\newcommand{\FF}{{\mathbf{F}}}
\newcommand{\PP}{{\mathbf{P}}}
\newcommand{\ZZ}{{\mathbf{Z}}}

\newcommand{\Fq}{\FF_q}
\newcommand{\kbar}{\bar{k}}
\newcommand{\pibar}{\bar{\pi}}
\newcommand{\manypointslink}{\href{http://www.manypoints.org}{\texttt{manypoints}}}

\renewcommand{\hat}{\widehat}
\renewcommand{\tilde}{\widetilde}

\newcommand\lowtilde{\lower0.7ex\hbox{\textasciitilde}}

\newcommand{\mybar}[1]{
  \mathchoice
  {#1\llap{$\overline{\phantom{\displaystyle\rm#1}}$}}
  {#1\llap{$\overline{\phantom{\textstyle\rm#1}}$}}
  {#1\llap{$\overline{\phantom{\scriptstyle\rm#1}}$}}
  {#1\llap{$\overline{\phantom{\scriptscriptstyle\rm#1}}$}}
}  
\renewcommand{\bar}{\mybar}

\begin{document}

\title[Curves of medium genus]{Curves of medium genus with many points}

\author{Everett W. Howe} 
\address{Center for Communications Research,
         4320 Westerra Court,
         San Diego, CA 92121, USA}
\email{however@alumni.caltech.edu}
\urladdr{http://www.alumni.caltech.edu/\lowtilde{}however/}

\date{29 March 2017}


\begin{abstract}
The \emph{defect} of a curve over a finite field is the difference between the 
number of rational points on the curve and the Weil--Serre upper bound for the 
number of points on the curve.  We present algorithms for constructing curves of
genus~$5$, $6$, and $7$ with small defect. Our aim is to be able to produce, in
a reasonable amount of time, curves that can be used to populate the online 
table of curves with many points found at 
\href{http://www.manypoints.org}{\texttt{manypoints.org}}.
\end{abstract}

\maketitle

\section{Introduction}
\label{sec:intro}

For every prime power $q$ and non-negative integer $g$, we let $N_q(g)$ denote
the maximum number of rational points on a smooth, projective, absolutely
irreducible curve of genus $g$ over the finite field~$\Fq$.  At the turn of the
present century,  van der Geer and van der Vlugt published  
tables~\cite{GeerVlugt2000} of the best upper and lower bounds on $N_q(g)$ known
at the time, for $g\le 50$ and for $q$ ranging over small powers of $2$ and~$3$.
In~2010, van der Geer, Lauter, Ritzenthaler, and the author (with technical
assistance from Gerrit Oomens) created the \manypointslink\ web 
site~\cite{manypoints}, which gives the currently-known best upper and lower 
bounds on $N_q(g)$ for $g\le 50$ and for a range of prime powers $q$: the 
primes less than~$100$, the prime powers $p^i$ for $p<20$ and $i\le 5$, 
and the powers of $2$ up to~$2^7$.

The \emph{Weil--Serre upper bound}~\cite{Serre1983a} for $N_q(g)$ states that
\[
N_q(g) \le q + 1 + g \lfloor 2\sqrt{q}\rfloor.
\]
When $q \ge (g + \sqrt{g+1})^2$ the Weil--Serre bound is almost always the best
upper bound currently known for $N_q(g)$; the exceptions come from 
``exceptional'' prime powers~\cite[Theorem~4, p.~1682]{HoweLauter2003} and from 
careful case-by-case analyses --- see the introduction to~\cite{HoweLauter2012}
for a summary. On the other hand, lower bounds for $N_q(g)$ are generally 
obtained by producing more-or-less explicit examples of curves with many points.
Typically, this is done by searching through specific families of curves --- for
instance, families of curves obtained via class field theory as covers of 
lower-genus curves, or families of curves with specific nontrivial automorphism
groups.  For small finite fields of characteristic $2$ and $3$, such searches
have been carried out for many genera, so even the earliest versions of the 
van der Geer--van der Vlugt tables gave nontrivial lower bounds for many values 
of~$N_q(g)$.

One of the goals of the \manypointslink\ web site is to encourage researchers to
consider curves over finite fields of larger characteristics. Curves over these 
fields have received far less attention than curves in characteristic $2$ 
or~$3$, so at present the table entries for lower bounds for $N_q(g)$ remain
unpopulated for most $q$ and~$g$.

Of course, there are trivial lower bounds for $N_q(g)$: for instance, if $C$
is a hyperelliptic curve of any genus over $\Fq$, then either $C$ or its 
quadratic twist will have at least $q+1$ points.  To avoid having to worry
about such ``poor'' lower bounds, van der Geer and van der Vlugt decided not to
print a lower bound for $N_q(g)$ in their tables unless it was greater than 
$1/\sqrt{2}$ times the best upper bound known for $N_q(g)$.  This restriction
was adopted for the \manypointslink\ table as well, but it turns out that
it is not a strong enough filter when $q$ is large with respect to~$g$.  The 
administrators of the \manypointslink\ site are considering replacing it with 
the requirement that a lower bound not be published unless the difference
between the lower bound and $q+1$ is at least $1/\sqrt{2}$ times the difference
between the best proven upper bound and $q+1$.

Every genus-$0$ curve over a finite field is isomorphic to $\PP^1$, so
$N_q(0) = q+1$ for all~$q$. For $g=1$ and $g=2$, the value of $N_q(g)$ is also 
known for all $q$.  For $g=1$ this is due to a classical result of 
Deuring~\cite{Deuring1941} (see~\cite[Theorem~4.1, p.~536]{Waterhouse1969}); 
for $g=2$, this is due to work of Serre~\cite{Serre1983a,Serre1983b,Serre1984} 
(see also~\cite{HoweNartEtAl2009}).  There is no easy formula known 
for~$N_q(3)$, but for all $q$ in the \manypointslink\ table the value has
been computed; see the introduction to~\cite{Mestre2010}.  For genus~$4$, the 
exact value of $N_q(4)$ is known for $43$ of the $59$ prime powers $q$ in the 
\manypointslink\ table, and for the remaining $16$ prime powers the lower 
bound for $N_q(4)$ is within $4$ of the best proven upper bound; 
see~\cite{Howe2012} and~\cite{Howe2016}.

In this paper we develop methods for finding curves of genus $5$, $6$, and $7$
with reasonably many points.  Our goal is to find lower bounds for $N_q(g)$ for 
these genera that are somewhat close to the best proven upper bounds, in order
to ``raise the bar'' for what counts as an interesting example.  We hope that
this will inspire others to think of new constructions, new search strategies, 
and faster algorithms.  

The \emph{defect} of a curve $C$ of genus $g$ over $\Fq$ is the difference
between $\#C(\Fq)$ and the Weil--Serre upper bound for genus-$g$ curves 
over~$\Fq$. If $\pi_1, \pibar_1, \ldots, \pi_g, \pibar_g$ are the Frobenius 
eigenvalues (with multiplicity) of $C$, listed in complex-conjugate pairs,
then we have 
\[
\#C(\Fq) = q + 1 - (\pi_1 + \pibar_1 + \cdots + \pi_g + \pibar_g),
\]
so the defect of $C$ is given by
\[
(q + 1 + g \lfloor 2\sqrt{q}\rfloor) - \#C(\Fq) 
= \sum_{i=1}^g \left(\lfloor 2\sqrt{q}\rfloor + \pi_i + \pibar_i\right).
\]
If the Jacobian of $C$ decomposes up to isogeny as the product of the Jacobians
of some other curves~$C_i$, then the multiset of Frobenius eigenvalues for $C$ 
is the union of the multisets of Frobenius eigenvalues for the~$C_i$, and it 
follows that the defect of $C$ is the sum of the defects of the~$C_i$.

The ideas behind our new constructions for producing curves of genus $5$, $6$, 
and $7$ with small defect are similar to those used in~\cite{Howe2016} to 
produce curves of genus $4$ with small defect. The basic strategy is to search 
through families of curves $D$ that are Galois extensions of $\PP^1$ with group
$G$ isomorphic to a $2$-torsion group; in practice, we will consider 
$G\cong (\ZZ/2\ZZ)^2$ and $G\cong (\ZZ/2\ZZ)^3$.  The extension $D\to \PP^1$ 
will then have $\#G-1$ subextensions $C_i\to\PP^1$ of degree~$2$, corresponding
to the index-$2$ subgroups of~$G$. A result of Kani and 
Rosen~\cite[Theorem B, p.~308]{KaniRosen1989} says that in this situation, the 
Jacobian of $D$ is isogenous to the product of the Jacobians of the~$C_i$; it
follows that the genus of $D$ is the sum of the genera of the~$C_i$, and the 
defect of $D$ is equal to the sum of the defects of the~$C_i$. 

Within this basic framework, however, there are many ways of structuring an 
algorithm, some of which are much more efficient than others. In 
Section~\ref{sec:4} we review the genus-$4$ construction used in our earlier
paper~\cite{Howe2016}, and we describe our new constructions for higher genera
in Sections~\ref{sec:5}, \ref{sec:6}, and~\ref{sec:7}.

At various points in this paper we will find it convenient to consider the 
\emph{minimal defect} $D_q(g)$ for genus-$g$ curves over~$\Fq$, which we define
by
\[
D_q(g) = q + 1 + g \lfloor 2\sqrt{q}\rfloor - N_q(g).
\]
Giving a lower bound for $N_q(g)$ is equivalent to giving an upper bound
for~$D_q(g)$.

The computations we describe in this paper were implemented in
Magma~\cite{BosmaCannonEtAl1993} on a 2010-era laptop computer, with a 2.66~GHz
Intel Core i7 ``Arrandale'' processor and 8~GB of RAM.

\section{Review of the genus-\texorpdfstring{$4$}{4} construction}
\label{sec:4}

In~\cite{Howe2016}, we presented an algorithm to produce genus-$4$ curves $D$
of small defect over finite fields $\Fq$ of odd characteristic. The heuristic
argument presented in~\cite{Howe2016} leads us to expect that as $q$ ranges over
all sufficient large primes and over almost all prime powers, the algorithm will
produce a genus-$4$ curve with defect at most $4$ in time $\Otilde(q^{4/3})$. 
The construction has two main ideas: First, if one is given a genus-$2$ curve 
$C/\Fq$ that can be written $y^2 = f_1 f_2$ for two cubic polynomials $f_1$ and
$f_2$ in $\Fq[x]$, then one can efficiently find a value of $a\in \Fq$, if such 
a value exists, such that the two genus-$1$ curves $E_1$ and $E_2$ defined by 
$y^2 = (x-a) f_1$ and $y^2 = (x-a) f_2$ have small defect. If such an $a$ 
exists, and if we let $D$ be the curve defined by the two equations 
$w^2 = (x-a) f_1$ and $z^2 = (x-a) f_2$, then the degree-$4$ cover $D\to\PP^1$
that sends $(x,w,z)$ to $x$ is a Galois extension, with group $(\ZZ/2\ZZ)^2$,
and we have the following diagram, in which each arrow denotes a degree-$2$ map:
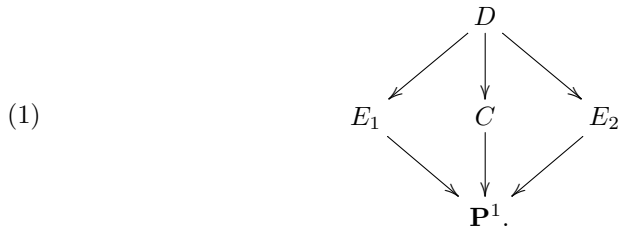
\begin{equation}
\label{EQ:V4}
\begin{gathered}
\xymatrix{
              & D\ar[dl]\ar[dr]     &             \\
 E_1\ar[dr]   & C \ar@{<-}[u]\ar[d] & E_2\ar[dl]  \\
              & \ \PP^1.            &             \\
}
\end{gathered}
\end{equation}
This is the situation described in Section~\ref{sec:intro}, and we find that the
genus of $D$ is $4$ and the defect of $D$ is equal to the sum of the defects of
$E_1$, $E_2$, and~$C$.  In particular, since $E_1$ and $E_2$ have small defect,
we see that the defect of $D$ is not much larger than the defect of $C$.

The second main idea in the construction of~\cite{Howe2016} is to have a method
for efficiently producing genus-$2$ curves $C$ with small defect. 
In~\cite{Howe2016} this is accomplished by taking pairs of elliptic curves with
small defect, ``gluing'' them together along their $2$- or $3$-torsion subgroups 
using~\cite[Proposition~4, p.~325]{HoweLeprevostEtAl2000} 
and~\cite[Algorithm~5.4, p.~185]{BrokerHoweEtAl2015} to produce genus-$2$
curves with small defect, and then using Richelot isogenies to produce more and 
more curves with small defect from these seed curves.

The algorithm to produce genus-$4$ curves with small defect works by producing
many genus-$2$ curves with small defect (using the second idea) until one is 
produced for which one can find a value of $a$ (using the first idea) that leads
to a diagram like Diagram~\eqref{EQ:V4}.

We note for future reference that the discussion leading up to Heuristic~5.6 
in~\cite{Howe2016} suggests that for fixed $d$ and for $q\to\infty$ we might 
expect there to be on the order of $d^{3/2} q^{3/4}$ genus-$2$ curves of 
defect~$d$.  Therefore, we might also expect there to be on the order of 
$d^{5/2} q^{3/4}$ genus-$2$ curves of defect $d$ or less.  Similarly, we might
expect there to be on the order of $d^{3/2} q^{1/4}$ elliptic curves of defect 
$d$ or less.  Since there are about $2q$ elliptic curves over $\Fq$ and about 
$2q^3$ curves of genus~$2$, we see that it is reasonable to expect that a random
elliptic curve will have defect at most $d$ with probability on the order of 
$d^{3/2}/ q^{3/4}$, while a random genus-$2$ curve will have defect at most $d$ 
with probability on the order of $d^{5/2}/q^{9/4}$.

\section{Constructions for genus-\texorpdfstring{$5$}{5} curves}
\label{sec:5}

Before we describe our algorithms for constructing genus-$5$ curves with small
defect, we present a technique for efficiently computing all separable quartics 
and cubics over a finite field~$k$, up to squares in $k^*$ and the action 
of~$\PGL_2(k)$.  

Let $k = \Fq$ be a finite field of odd characteristic, and let $S$ denote the 
set of separable quartics and cubics in $k[x]$.  The group of squares in $k^*$ 
acts on $S$ by multiplication, and the group $\PGL_2(k)$ acts on~$S$ modulo 
squares as follows:  Given $f\in S/k^{*2}$ and $\alpha\in\PGL_2(k)$, we let 
$\left[\begin{smallmatrix}a&b\\c&d\end{smallmatrix}\right]$ be a matrix that 
represents $\alpha$, and we define 
\[
\alpha(f \bmod k^{*2}) = (cx+d)^4 f\left(\frac{ax+b}{cx+d}\right)\bmod k^{*2}.
\]
Had we chosen a different matrix to represent $\alpha$, the right-hand side of 
the preceding equality would be modified by a square, so we do get a 
well-defined action of $\PGL_2(k)$ on $S/k^{*2}$.  In one of our algorithms it
will be useful to be able to quickly calculate orbit representatives for $S$ 
under this combined action.

Given a separable quartic or cubic $f\in k[x]$, let $C$ be the curve $y^2 = f$
and let $E$ be the Jacobian of~$C$.  Since $C$ has genus~$1$, $E$ is an elliptic
curve, and since genus-$1$ curves over $k$ have rational points, there is an
isomorphism $\varphi\colon C\to E$.  Under this isomorphism, the involution 
$\iota$ of $C$ given by $(x,y)\mapsto (x,-y)$ becomes an involution on~$E$, 
which must be of the form $P\mapsto P_0 - P$ for some point $P_0$ in~$E(k)$.
If $\varphi'\colon C\to E$ is another isomorphism, then there is an automorphism
$\varepsilon$ of $E$ and a point $Q_0\in E(k)$ such that 
\[
\varphi'(P) = \varepsilon\varphi(P) + Q_0,
\]
so that $\varphi'$ takes the involution $\iota$ on $C$ to the involution
$P\mapsto P'_0 - P$ on $E$, where $P'_0 = \varepsilon(P_0) + 2Q_0$. Thus, given 
$f\in S$, we obtain a pair $(E,[P_0])$, where $E$ is an elliptic curve over~$k$ 
and $[P_0]$ is an element of $E(k)/2E(k)$ up to the action of $\Aut E$.

\begin{lemma*}
The map from $S$ to pairs $(E,[P_0])$ defined above induces a bijection between
the orbits of $S/k^{*2}$ under the action of $\PGL_2(k)$ and the set of all 
pairs $(E,[P_0])$.
\end{lemma*}

\begin{proof}
The pair $(E,[P_0])$ that we obtain from $f$ clearly depends only on the
isomorphism class of the pair $(C,\iota)$, and that isomorphism class is fixed
by the actions of $k^{*2}$ and $\PGL_2(k)$.  Thus, we do indeed get a map from 
orbits to pairs.

Suppose $f_1$ and $f_2$ are two elements of $S$ that give rise to the same pair
$(E,[P_0])$.  Let $C_1$ and $C_2$ be the curves $y^2 = f_1$ and $y^2 = f_2$,
respectively, with involutions $\iota_1$ and~$\iota_2$. Then there are 
isomorphisms $\varphi_1\colon C_1\to E$ and $\varphi_2\colon C_2\to E$ that take
$\iota_1$ and $\iota_2$ to the involution $P\mapsto P_0 - P$ of~$E$, so the
isomorphism $\psi = \varphi_2^{-1} \varphi_1$ from $C_1$ to $C_2$ takes 
$\iota_1$ to $\iota_2$.  It follows that $\psi$ is of the form
\[
(x,y) \mapsto \left(\frac{ax + b}{cx+d}, \frac{e y}{(cx + d)^2}\right)
\]
for some $a,b,c,d,e\in k$.  We see that
$f_1 \bmod k^{*2} = \alpha(f_2 \bmod k^{*2}),$ where $\alpha$ is the class of 
$\left[\begin{smallmatrix}a&b\\c&d\end{smallmatrix}\right]$ in $\PGL_2(k)$.
Thus, the map from orbits to pairs in injective.

To finish the proof, we need to show that the map from orbits to pairs is
surjective.  Suppose $E$ is an elliptic curve over $k$ and $P_0$ is a point 
in~$E(k)$.  We will produce a cubic or quartic $f$ that gives rise to the pair
$(E,[P_0])$.

Let $y^2=g$ be a Weierstrass model for $E$, where $g\in k[x]$ is a monic cubic.
If $P_0$ is the infinite point on $E$ then we can just take $f = g$, so assume
that $P_0$ is an affine point $(x_0,y_0)$.  By replacing $g(x)$ with $g(x+x_0)$,
we may assume that $x_0 = 0$, so that $g$ is of the form $x^3 + ax^2 + bx + c^2$
and $P_0 = (0,c)$.

Let $f = x^4 - 2a x^2 - 8cx + a^2 - 4b$.  We compute that the discriminant of
$f$ is $2^{12}$ times the discriminant of $g$, so $f$ is a separable quartic. 
If we let $C$ be the curve $y^2 = f$, then there is an isomorphism 
$\psi\colon C\to E$ given by
\[
(x,y) \mapsto \left( \frac{y + x^2 - a}{2}, \frac{x( y + x^2 - a)}{2} - c\right),
\]
and we compute that the two infinite points on $C$ are sent to the points $P_0$
and $\infty$ on $E$.  It follows that the involution $\iota$, which swaps the 
two infinite points on $C$, gets sent to the involution $P\mapsto P_0 - P$,
which swaps $P_0$ and $\infty$.  Thus every pair $(E,[P_0])$ comes from an 
element of~$S$.
\end{proof}

Note that the proof of this lemma gives us an algorithm for producing 
representatives for all of the orbits of $S/k^{*2}$ under the action of 
$\PGL_2(k)$, in time~$\Otilde(q)$: For each $j$-invariant in $k$, we list all
the isomorphism classes of elliptic curves $E$ with that $j$-invariant, compute
the group $E(k)/2E(k)$ up to the action of $\Aut E$, and for a representative
point $P_0$ for each element we compute the quartic $f$ that appears at the end
of the proof of the lemma.

We turn now to our constructions of small-defect curves of genus~$5$ over a 
finite field $k = \Fq$ of odd characteristic.  Our strategy for producing such
curves will be to construct diagrams like Diagram~\eqref{EQ:V4}, except that the
genera of the intermediate curves will be $2$, $1$, and~$2$, instead of $1$, 
$2$, and~$1$. In other words, our aim will be to construct a diagram
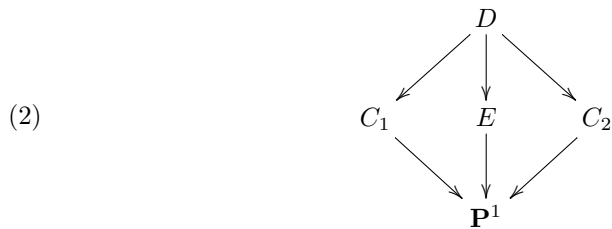
\begin{equation}
\label{EQ:genus5}
\begin{gathered}
\xymatrix{
            & D\ar[dl]\ar[dr]     &             \\
 C_1\ar[dr] & E \ar@{<-}[u]\ar[d] & C_2\ar[dl]  \\
            & \PP^1               &             \\
}
\end{gathered}
\end{equation}
in which $C_1$ and $C_2$ are genus-$2$ curves with small defect, $E$ is a 
genus-$1$ curve with small defect, and the arrows are maps of degree~$2$.  This
means that we would like to find genus-$2$ curves $w^2 = f g_1$ and 
$z^2 = f g_2$, where $f$ is a quartic or a cubic polynomial and $g_1$ and $g_2$
are coprime quadratics, such that both of the genus-$2$ curves have small 
defect, and such that the genus-$1$ curve $y^2 = g_1 g_2$ also has small defect.
(If $f$ is a quartic, we can also allow one of $g_1$ and $g_2$ to be linear.)

Our first strategy runs as follows.  We begin by enumerating separable quartics
and cubics $f\in k[x]$ up to the action of $k^{*2}$ and $\PGL_2(k)$, as outlined
above. For each such $f$ we enumerate all quadratic and linear polynomials~$g$,
up to squares in~$k^*$, such that $fg$ is separable of degree $5$ or $6$ and 
such that $y^2 = fg$ has small defect.  We then consider all pairs $(g_1,g_2)$
of such $g$ such that $g_1 g_2$ is separable of degree $3$ or~$4$, and we check 
whether $y^2 = g_1 g_2$ has small defect. The meaning of ``having small defect''
will change dynamically as we run the algorithm, depending on the smallest 
defect we have found so far for a triple~$(f,g_1,g_2)$.

This first strategy works well for small~$q$, and it is guaranteed to find the 
genus-$5$ curve of smallest defect that fits into a diagram like 
Diagram~\eqref{EQ:genus5}.  However, there are on the order of $q$ quartics and
cubics $f$ to consider, and for each $f$ we have to enumerate on the order of 
$q^2$ quadratics and linears~$g$, and for each $(f,g)$ pair we have to compute 
the number of points on a genus-$2$ curve.  Assuming we count points na\"ively, 
this means that even just this portion of our first strategy will already take
time roughly on the order of~$q^4$.

For larger fields, therefore, we use an alternate strategy. To produce our pairs
of genus-$2$ curves, we will enumerate many genus-$2$ curves with small defect, 
keeping track of the ones that can be written $y^2 = h$ with $h$ the product of
a quartic and a quadratic in $k[x]$. Suppose we have two such curves, 
$w^2 = f_1 g_1$ and $z^2 = f_2 g_2$.  We would like to be able to tell whether
a change of coordinates (via a linear fractional transformation in $x$) could
transform $f_2$ into a constant times~$f_1$. There are two necessary conditions
for this to happen: First, the degrees of the irreducible factors of $f_1$ must 
match those of $f_2$, and second, the $j$-invariants of the genus-$1$ curves 
$y^2 = f_1$ and $y^2 = f_2$ must be equal. These two conditions are not quite 
sufficient --- there's only one chance in three that two different products of 
irreducible quadratics with the same $j$-invariant can be transformed to 
constant multiples of one another via a linear fractional transformation, and 
there are additional complications for curves with $j$-invariant $0$ or~$1728$
--- but for our purposes these necessary conditions will be good enough as a 
first test.

So our second strategy will be to enumerate genus-$2$ curves with small defect
that can be written $y^2 = f g $ with $f$ a quartic and $g$ a quadratic, and 
keep a list of the curves together with the $j$-invariants and factorization 
degrees of the associated quartics~$f$.  Whenever a $j$-invariant and set of 
factorization degrees occurs twice, and the associated quartics can be 
transformed into one another, we will have found two genus-$2$ curves 
$y^2 = f_1 g_1$ and $y^2 = f_2 g_2$ that can be put into 
Diagram~\eqref{EQ:genus5} (provided that the linear fractional transformation 
that takes $f_2$ to $f_1$ does not take $g_2$ to a quadratic with a factor in 
common with $g_1$). Then we need to compute the number of points on the 
genus-$1$ curve in the middle of the diagram to see whether it has small 
defect.\footnote{
     Warning: On versions of Magma at least up to V2.22-3, 
     using \texttt{\#HyperellipticCurve(f)} to count the 
     number of points on $y^2 = f$ will produce incorrect 
     results when $f$ is a quartic with nonsquare leading 
     coefficient over a finite field with more then $10^6$
     elements.  This affects us when $q = 17^5$ or $q = 19^5$.}

To produce genus-$2$ curves with small defect, we use the technique mentioned
in the previous section: We ``glue together'' pairs of elliptic curves with
small defect, and then use Richelot isogenies to find more and more genus-$2$
curves with Jacobians isogenous to the product of the given elliptic curves.

As we continue to compute, we will find more and more examples.  If at a
certain point we have found a genus-$5$ curve with defect $d$, then we
know that we can stop looking for better examples once we have examined
all of the genus-$2$ curves with defect at most~$d$ that we can construct using
the technique described above.  If we continue to run the algorithm until this
has happened, we say that we have \emph{run the algorithm to completion}.  
Running to completion simply means that we have reached a point when we know
that continuing to run the algorithm will not produce any examples better than
what we have already found.


We ran the first algorithm for all of the odd prime powers $q$ listed in the
\manypointslink\ table with $q\le 19^2$.  On the modest laptop described in 
Section~\ref{sec:intro}, the computation took $50$ minutes for $q = 17^2$ and 
$108$ minutes for $q = 19^2$.  For the larger $q$ in the \manypointslink\ table,
we were able to run the second algorithm to completion; the computation for
$q = 19^5$ took nearly $30$ hours. In Table~\ref{table:5} we present the current 
lower and upper bounds on the minimal defect $D_q(5)$ for the odd values of $q$ 
listed in the online table, with the upper bound in boldface if it was obtained
from our computations.  Equations for the curves we found can be obtained from
the \manypointslink\ site by clicking on the appropriate table entry.

\begin{table}
\begin{center}
\begin{tabular}{llrllrllrllr}
\toprule
$q$ & range &&$q$ & range &&$q$ & range &&$q$ & range \\
\cmidrule(lr){1-2}\cmidrule(lr){4-5}\cmidrule(lr){7-8}\cmidrule(lr){10-11}
$3$   & $6$--$6$           &&$11$   & $4$--$5$           &&$19$   & $3$--$\mathbf{6}$  &&$67$ & $0$--$\mathbf{8}$ \\
$3^2$ & $5$--$8$           &&$11^2$ & $0$--$0$           &&$19^2$ & $0$--$0$           &&$71$ & $0$--$0$          \\
$3^3$ & $3$--$6$           &&$11^3$ & $0$--$\mathbf{0}$  &&$19^3$ & $0$--$\mathbf{5}$  &&$73$ & $3$--$\mathbf{9}$ \\
$3^4$ & $0$--$\mathbf{4}$  &&$11^4$ & $0$--$0$           &&$19^4$ & $0$--$\mathbf{4}$  &&$79$ & $0$--$\mathbf{5}$ \\
$3^5$ & $2$--$\mathbf{13}$ &&$11^5$ & $0$--$0$           &&$19^5$ & $0$--$\mathbf{15}$ &&$83$ & $2$--$\mathbf{8}$ \\
$5$   & $6$--$6$           &&$13$   & $5$--$9$           &&$23$   & $3$--$\mathbf{5}$  &&$89$ & $0$--$\mathbf{4}$ \\
$5^2$ & $4$--$\mathbf{8}$  &&$13^2$ & $0$--$\mathbf{4}$  &&$29$   & $2$--$\mathbf{4}$  &&$97$ & $0$--$\mathbf{5}$ \\
$5^3$ & $3$--$\mathbf{12}$ &&$13^3$ & $0$--$\mathbf{9}$  &&$31$   & $5$--$5$           \\
$5^4$ & $0$--$\mathbf{8}$  &&$13^4$ & $0$--$\mathbf{8}$  &&$37$   & $4$--$\mathbf{8}$  \\
$5^5$ & $0$--$\mathbf{17}$ &&$13^5$ & $0$--$\mathbf{8}$  &&$41$   & $0$--$5$           \\
$7$   & $5$--$7$           &&$17$   & $5$--$\mathbf{8}$  &&$43$   & $3$--$\mathbf{9}$  \\
$7^2$ & $0$--$0$           &&$17^2$ & $0$--$\mathbf{4}$  &&$47$   & $0$--$\mathbf{5}$  \\
$7^3$ & $3$--$\mathbf{9}$  &&$17^3$ & $2$--$\mathbf{10}$ &&$53$   & $0$--$\mathbf{4}$  \\
$7^4$ & $0$--$\mathbf{8}$  &&$17^4$ & $0$--$\mathbf{4}$  &&$59$   & $2$--$\mathbf{7}$  \\
$7^5$ & $0$--$\mathbf{15}$ &&$17^5$ & $0$--$\mathbf{5}$  &&$61$   & $0$--$5$           \\
\bottomrule
\end{tabular}
\end{center}
\bigskip
\caption{The best upper and lower bounds known for the minimal defect $D_q(5)$ 
for the odd values of $q$ represented in the \manypointslink\ table, as of 
29 March 2017. The upper bounds in boldface come from computations described
in this paper.}
\label{table:5}
\end{table}

One could also try to construct curves of genus $5$ by considering diagrams like
Diagram~\eqref{EQ:V4} in which the genera of the intermediate curves are $1$, 
$3$, and~$1$; Soomro~\cite{Soomro2013} takes this approach.

\section{Constructions for genus-\texorpdfstring{$6$}{6} curves}
\label{sec:6}

We work over a finite field $k = \Fq$ of odd characteristic.  The genus-$6$ 
curves that we will construct will fit into a diagram
\begin{equation}
\label{EQ:genus6}
\begin{gathered}
\xymatrix{
            & D\ar[dl]\ar[dr]       &             \\
 C_1\ar[dr] & C_2 \ar@{<-}[u]\ar[d] & C_3\ar[dl]  \\
            &\ \PP^1,               &             \\
}
\end{gathered}
\end{equation}
where the curves $C_1$, $C_2$, and $C_3$ have genus~$2$ and where the arrows
represent degree-$2$ maps.  Then the defect of $D$ will be the sum of the
defects of the $C_i$.  To produce such a diagram, we need to find three cubic
polynomials $f_1$, $f_2$, and $f_3$ such that each curve $C_i$ is given by 
$y^2 = \prod_{j\ne i} f_j$. (Actually, one of the $f_i$ could be a quadratic;
this will be the case when two of the $C_i$ have a Weierstrass point at
infinity.) And, of course, we want the $C_i$ to have small defect.

Before we describe how to construct good triples $(f_1,f_2,f_3)$, we note that 
we can define an action of $\PGL_2(k)$ on the set of separable cubics and 
quadratics in~$k[x]$, up to multiplicative constants in $k^*$, analogous to the
action we defined in Section~\ref{sec:5} for quartics and cubics up to squares 
in~$k^*$.  Given a cubic or quadratic $f\in k[x]/k^*$ and $\alpha\in\PGL_2(k)$, 
we let $\left[\begin{smallmatrix}a&b\\c&d\end{smallmatrix}\right]$ be a matrix 
that represents $\alpha$, and we define 
\[
\alpha(f \bmod k^*) = (cx+d)^3 f\left(\frac{ax+b}{cx+d}\right) \bmod k^*.
\]
Since everything is defined only up to $k^*$, this does not depend on the 
choice of representative for~$\alpha$.  Now we set $g_1 = x(x-1)$, and we fix
irreducible quadratic and cubic polynomials $g_2$ and~$g_3$.  It is easy to see 
from the $3$-transitive action of $\PGL_2(\kbar)$ on $\PP^1(\kbar)$ that every
separable cubic or quadratic polynomial $h\in k[x]$ can be transformed by an
element of $\PGL_2(k)$ into a constant times one of the polynomials $g_1$, 
$g_2$, or~$g_3$, depending on the degrees of the irreducible factors of~$h$.

We construct triples $(f_1,f_2,f_3)$ as follows.  We start enumerating 
genus-$2$ curves $C$ with small defect, and we keep track of the ones that can
be written $y^2 = h_1 h_2$, where $h_1$ and $h_2$ are cubics.  For every such 
representation of $C$, we apply a linear fractional transformation that takes
$h_1$ to a constant times one of our three fixed polynomials $g_1$, 
$g_2$,~$g_3$.  That means we can write $C$ as $y^2 = g_i f$ for some cubic 
(or, in some circumstances, quadratic) polynomial~$f$.  Whenever we add a new 
pair $(g_i, f)$ to our growing list, we look at all other pairs $(g_i, \hat{f})$
already on our list, and we check to see whether $y^2 = f \hat{f}$ is a curve 
with small defect. If so, by setting $f_1 = f$ and $f_2 = \hat{f}$ and 
$f_3 = g_i$ we have a triple $(f_1,f_2,f_3)$ that gives us a genus-$6$ curve
with small defect.

We continue enumerating genus-$2$ curves $C$ with small defect until we reach 
the defect of the current record-holding triple $(f_1,f_2,f_3)$.  At this point 
we are guaranteed that we will find no curves of smaller defect by using this
construction, and again we say that we have run the algorithm to completion.

One way of producing genus-$2$ curves of small defect is simply to make a list 
of all curves of the form $y^2 = g_i f$ for $i=1,2,3$ and for $f$ ranging over
the separable cubic and quadratic polynomials in $k[x]$, up to squares in~$k^*$.  
Then one can compute the defects of the curves in the list, and sort the list 
accordingly.

If the field $k$ is too large for this to be feasible, and if $k$ has a proper
subfield, one can instead consider only the cubics and quadratics $f$ with
coefficients in the subfield.  This is much faster, but of course results in
missing many possible curves that might have small defect.

A third possibility is to use the procedure already described in earlier 
sections: gluing together pairs of elliptic curves of small defect, and
using Richelot isogenies to obtain even more genus-$2$ curves.

And finally, we have a fourth algorithm with a slightly different flavor, which 
depends on the observation that a random element of $\PGL_2(\Fq)$ has order~$3$ 
with probability approximately $1/q$.  (More precisely, by computing the 
centralizer of the image of 
$\left[\begin{smallmatrix} 0 & -1 \\ 1 & \phantom{-}1\end{smallmatrix}\right]$ 
in $\PGL_2(\Fq)$, one shows that the fraction of elements of $\PGL_2(\Fq)$ of 
order $3$ is equal to either $1/q$ or $1/(q-1)$ or $1/(q+1)$, depending on 
whether $q$ is congruent to $0$, $1$, or $2$ modulo~$3$.)

We enumerate genus-$2$ curves with small defect by gluing together elliptic 
curves and using Richelot isogenies.  Suppose $C$ is such a curve, say given by 
an equation $y^2 = g$, where $g$ is a sextic polynomial, and suppose we can 
write $g$ as $f_1 f_2$, where $f_1$ and $f_2$ are cubics. We compute all of the 
linear fractional transformations $\mu$ that take the roots of $f_2$ to the
roots of~$f_1$.  Suppose further that one such $\mu$ has order $3$ in 
$\PGL_2(k)$; if we assume that $\mu$ behaves like a random element 
of~$\PGL_2(k)$, this will happen with probability about~$1/q$.  In this case, 
we write $\mu = (ax+b)/(cx+d)$; the condition that $\mu$ has order $3$ means 
that $a^2 + ad + d^2 + bc = 0$.  Let $e$ be the constant such that 
$f_2(x) = e f_1(\mu)(cx + d)^3$, and set $f_3(x) = e f_2(\mu)(cx + d)^3$.
Note that then
\[ e f_3(\mu) (cx + d)^3 = - e^3 (a + d)^9 f_1(x).\]
We see that change of variables $(x,y) \mapsto (\mu, e^{-1}y/(cx+d)^3)$ gives an 
isomorphism from the curve $y^2 = f_2 f_3$ to the curve $y^2 = f_1 f_2$, and if
$-e(a+d)^3$ is a square, say $-e(a+d)^3=s^2$, then 
$(x,y) \mapsto (\mu, e^{-1}s^3 y / (cx+d)^3)$ gives an isomorphism from
$y^2 = f_3 f_1$ to $y^2 = f_2 f_3$.  Thus, if $-e(a+d)^3$ is a square, the
triple $(f_1,f_2,f_3)$ gives us a diagram like Diagram~\eqref{EQ:genus6} in
which the three curves $C_1$, $C_2$, and $C_3$ are all isomorphic to the 
small-defect curve $C$ that we started with.  Then the genus-$6$ curve $D$ has 
defect equal to three times the defect of $C$.

Asymptotically, we expect this last algorithm to be much faster than the other 
three, because we are waiting on an event with probability roughly~$1/q$, rather
than the much rarer event that a random genus-$2$ curve has small defect.  As we 
noted in Section~\ref{sec:4}, the heuristics from~\cite{Howe2016} lead us to
expect that the probability that a random genus-$2$ curve over $\Fq$ has defect
at most $d$ grows like~$d^{5/2} q^{-9/4}$.


We implemented all four algorithms.  We applied the first to all odd $q<100$ 
listed in the \manypointslink\ table.  (For $q = 97$ this took $8$ minutes on
the laptop computer described in Section~\ref{sec:intro}.) We applied the 
second to all odd $q>100$ from the \manypointslink\ table, using the largest 
subfield of cardinality less than $100$.  (This took $2$ minutes for $q = 7^4$.)  
We applied the third method to the $q$ with $100 < q < 19^3$ and the fourth 
method to the $q$ with $q\ge 19^3$. (The third method took $370$ minutes for 
$q = 17^3$ and the fourth method took $15$ minutes for $q = 19^5$.  However, 
for $q = 17^5$ the fourth method took nearly $28$ hours, because the smallest 
defect found --- $108$ --- is quite large.) In Table~\ref{table:6} we present 
the current lower and upper bounds on the minimal defect $D_q(6)$ for the odd 
values of $q$ listed in the online table, with the upper bound in boldface if 
it was obtained from our computations. Equations for the curves we found can be
obtained from the \manypointslink\ site by clicking on the appropriate table 
entry.

\begin{table}
\begin{center}
\begin{tabular}{llrllrllrllr}
\toprule
$q$ & range &&$q$ & range &&$q$ & range &&$q$ & range \\
\cmidrule(lr){1-2}\cmidrule(lr){4-5}\cmidrule(lr){7-8}\cmidrule(lr){10-11}
$3$   & $8$--$8$           &&$11$   & $3$--$8$            &&$19$   & $3$--$\mathbf{8}$  &&$67$ & $0$--$12$         \\
$3^2$ & $8$--$11$          &&$11^2$ & $0$--$0$            &&$19^2$ & $0$--$0$           &&$71$ & $0$--$\mathbf{8}$ \\
$3^3$ & $4$--$12$          &&$11^3$ & $0$--$\mathbf{0}$   &&$19^3$ & $0$--$\mathbf{14}$ &&$73$ & $2$--$6$          \\
$3^4$ & $0$--$0$           &&$11^4$ & $0$--$\mathbf{6}$   &&$19^4$ & $0$--$\mathbf{6}$  &&$79$ & $0$--$6$          \\
$3^5$ & $2$--$\mathbf{14}$ &&$11^5$ & $0$--$\mathbf{0}$   &&$19^5$ & $0$--$\mathbf{18}$ &&$83$ & $2$--$12$         \\
$5$   & $6$--$8$           &&$13$   & $6$--$6$            &&$23$   & $3$--$\mathbf{8}$  &&$89$ & $0$--$\mathbf{4}$ \\
$5^2$ & $4$--$6$           &&$13^2$ & $0$--$0$            &&$29$   & $2$--$\mathbf{8}$  &&$97$ & $0$--$\mathbf{8}$ \\
$5^3$ & $3$--$\mathbf{14}$ &&$13^3$ & $0$--$\mathbf{14}$  &&$31$   & $6$--$\mathbf{10}$ \\
$5^4$ & $0$--$0$           &&$13^4$ & $0$--$0$            &&$37$   & $6$--$12$          \\
$5^5$ & $0$--$\mathbf{20}$ &&$13^5$ & $0$--$\mathbf{24}$  &&$41$   & $2$--$\mathbf{10}$ \\
$7$   & $6$--$\mathbf{10}$ &&$17$   & $6$--$\mathbf{10}$  &&$43$   & $6$--$\mathbf{14}$ \\
$7^2$ & $0$--$\mathbf{8}$  &&$17^2$ & $0$--$0$            &&$47$   & $3$--$6$           \\
$7^3$ & $3$--$\mathbf{16}$ &&$17^3$ & $2$--$\mathbf{12}$  &&$53$   & $3$--$6$           \\
$7^4$ & $0$--$0$           &&$17^4$ & $0$--$0$            &&$59$   & $3$--$\mathbf{10}$ \\
$7^5$ & $2$--$\mathbf{42}$ &&$17^5$ & $0$--$\mathbf{108}$ &&$61$   & $0$--$\mathbf{8}$  \\
\bottomrule
\end{tabular}
\end{center}
\bigskip
\caption{The best upper and lower bounds known for the minimal defect $D_q(6)$ 
for the odd values of $q$ represented in the \manypointslink\ table, as of 
29 March 2017. The upper bounds in boldface come from computations described 
in this paper.}
\label{table:6}
\end{table}

\section{Constructions for genus-\texorpdfstring{$7$}{7} curves}
\label{sec:7}

We work over a finite field $k = \Fq$ of odd characteristic. In our previous 
constructions, we produced degree-$4$ Galois extensions of~$\PP^1$ with group 
$(\ZZ/2\ZZ)^2$ by adjoining to $k(x)$ the square roots of two polynomials.  
For our genus-$7$ construction, we will instead produce degree-$8$ Galois 
extensions with group $(\ZZ/2\ZZ)^3$ by adjoining the square roots of three
polynomials $f_1$, $f_2$,~$f_3$.  According to the previously-cited result of 
Kani and Rosen~\cite[Theorem B, p.~308]{KaniRosen1989}, the resulting curve $D$
will have Jacobian isogenous to the product of the Jacobians of the seven
curves
\[
y^2 = f_1, \ 
y^2 = f_2, \ 
y^2 = f_3, \ 
y^2 = f_2 f_3, \ 
y^2 = f_1 f_3, \ 
y^2 = f_1 f_2, \ \text{and}\ 
y^2 = f_1 f_2 f_3,
\]
and the defect of $D$ will be the sum of the defects of these seven curves.  In 
order to produce a curve of genus~$7$, we will therefore want to have a method 
of choosing the polynomials $f_1$, $f_2$, and $f_3$ so that the sum of the 
genera of the seven associated curves is~$7$.

We used two methods to find such polynomials. The first method involves taking
\[
f_1 = s_1 (x-1) g_1, \qquad
f_2 = s_2 (x-1) g_2, \quad \text{and}\quad
f_3 = s x,
\]
where $g_1$ and $g_2$ are monic quadratic polynomials that are coprime to one
another and to $x-1$, and where $s$, $s_1$, and $s_2$ are nonzero constants that
only matter up to squares.  Up to squares in $k[x]$, the other polynomials we
then have to consider are
\begin{align*}
f_2f_3 &= s s_2 x (x-1) g_2\\
f_1f_3 &= s s_1 x (x-1) g_1 \\
f_1f_2/ (x-1)^2 &= s_1 s_2  g_1 g_2\\
f_1f_2f_3/(x-1)^2 &= s s_1 s_2 x  g_1 g_2.
\end{align*}
These seven polynomials give us hyperelliptic curves of genus $1$, $1$, $0$, 
$1$, $1$, $1$, and~$2$, respectively, so the curve $C$ will have genus~$7$.

To produce $f_1$, $f_2$, and $f_3$ of this form such that all of the seven
associated curves have small defect, we begin by enumerating the monic 
quadratic polynomials $h$ such that the genus-$1$ curves $y^2 = (x-1) h$ and 
$y^2 = x(x-1)h$ both have twists with small defect.  Then we let $g_1$ and $g_2$
range over the set of such $h$, and we check to see whether we can choose
constants $s$, $s_1$, and $s_2$ such that the curves defined by $f_1$, $f_2$, 
$f_1 f_3$, and $f_2 f_3$ simultaneously have small defect.  If we succeed in
doing so, we then check whether the genus-$1$ curve defined by $f_1 f_2$ has
small defect, and if it does, we check whether the genus-$2$ curve defined by 
$f_1 f_2 f_3$ has small defect.


We were able to run this algorithm to completion for the $q$ in the
\manypointslink\ table up to $17^3$.  On the laptop described in 
Section~\ref{sec:intro}, the computation for $q = 17^3$ took just over $29$ 
hours.

The second method involves taking 
\begin{align*}
f_1 &= s_1 x (x-1)       (x-b)              \\
f_2 &= s_2         (x-a) (x-b) (x-c)        \\
f_3 &= s_3 x       (x-a)             (x-d),
\end{align*}
where $a$, $b$, $c$, and $d$ are elements of $k$ that are distinct from one
another and from $0$ and $1$, and where $s_1$, $s_2$, and $s_3$ are nonzero
constants that only matter up to squares.  Up to squares in $k[x]$, the other 
polynomials we then have to consider are
\begin{align*}
f_2f_3    /   (x-a)^2       &=     s_2 s_3 x             (x-b) (x-c) (x-d) \\
f_1f_3    /  x^2            &= s_1     s_3   (x-1) (x-a)       (x-c) (x-d) \\
f_1f_2    /        (x-b)^2  &= s_1 s_2     x (x-1) (x-a)       (x-c)       \\
f_1f_2f_3 / (x(x-a)(x-b))^2 &= s_1 s_2 s_3   (x-1)             (x-c) (x-d).
\end{align*}
Each of these seven polynomials gives us a curve of genus $1$, so the curve $D$ 
will have genus~$7$.

(This construction of a $(\ZZ/2\ZZ)^3$-extension of $\PP^1$ of genus $7$ such 
that all the quadratic subextensions have genus $1$ can be viewed as a 
generalization of a method of constructing the Fricke--Macbeath 
curve~\cite{Macbeath1965}; the description of the Fricke--Macbeath curve 
discussed on pp.~533--534 of~\cite{Macbeath1965} makes this clear. In fact, 
if we take
\[
a =  - \zeta - \zeta^6, \qquad
b = \zeta^2 + \zeta^5, \qquad
c = \zeta^3 + \zeta^4 + 1, \quad \text{and\quad }
d = -\zeta - \zeta^6 -1,
\]
where $\zeta$ is a primitive $7$-th root of unity, we find that our curve is
geometrically isomorphic to the Fricke--Macbeath curve, because the linear 
fractional transformation $x \mapsto (x + \zeta)/(x + \zeta^{-1})$ sends the
set of ramification points $\{0,1,\infty,a,b,c,d\}$ to the set
$\{1,\zeta,\zeta^2,\zeta^3,\zeta^4,\zeta^5,\zeta^6\}$. Indeed, our construction
was inspired by seeing several records on the \manypointslink\ site produced by
Jaap Top and Carlo Verschoor by studying twists of the Fricke--Macbeath curve.)

Our strategy is to compute all the values of $\lambda\in k$ such that the
curve $y^2 = x(x-1)(x-\lambda)$, or its quadratic twist, has small defect.
Then we choose four such values $\lambda_1, \lambda_2, \lambda_3$, and 
$\lambda_4$, and compute the values of $a$, $b$, $c$, and $d$ such that the
curves defined by $y^2 = f_1$, $y^2 = f_3$, $y^2 = f_1 f_3$, and 
$y^2 = f_1 f_2 f_3$ have those $\lambda$-invariants (for a given ordering of
their $2$-torsion points).  Then we check to see whether we can choose the 
constants $s_1$, $s_2$, and $s_3$ so that the curves have small defect. If we 
succeed, we then check to see whether the remaining three curves have small 
defect.  We can do this quickly by first checking to see whether their 
$\lambda$-invariants are among those we computed at the beginning.

Of course, what we mean by ``small defect'' will change dynamically as we run 
our algorithm, depending on the curves we find.

We ran this algorithm to completion on all odd $q<17^5$ from the
\manypointslink\ list.  (For $q = 13^5$ this took nearly $50$ hours.) For 
$q = 17^5$ and $q=19^5$, we ran the algorithm for more than a week on a newer 
and faster desktop machine, but stopped with partial results and did not run to
completion. Combining the examples we found using both algorithms, we obtain the
results presented in Table~\ref{table:7}. The table gives the current lower and
upper bounds on the minimal defect $D_q(7)$ for the odd values of $q$ listed in 
the online table, with the upper bound in boldface if it was obtained from our 
computations. Equations for the curves we found can be obtained from the 
\manypointslink\ site by clicking on the appropriate table entry.

\begin{table}
\begin{center}
\begin{tabular}{llrllrllrllr}
\toprule
$q$ & range &&$q$ & range &&$q$ & range &&$q$ & range \\
\cmidrule(lr){1-2}\cmidrule(lr){4-5}\cmidrule(lr){7-8}\cmidrule(lr){10-11}
$3$   & $9$--$9$           &&$11$   & $4$--$\mathbf{10}$  &&$19$   & $0$--$\mathbf{12}$  &&$67$ & $4$--$\mathbf{16}$ \\
$3^2$ & $9$--$12$          &&$11^2$ & $0$--$0$            &&$19^2$ & $0$--$0$            &&$71$ & $0$--$\mathbf{8}$  \\
$3^3$ & $3$--$14$          &&$11^3$ & $0$--$\mathbf{24}$  &&$19^3$ & $0$--$\mathbf{23}$  &&$73$ & $6$--$\mathbf{21}$ \\
$3^4$ & $0$--$\mathbf{12}$ &&$11^4$ & $0$--$\mathbf{28}$  &&$19^4$ & $0$--$\mathbf{28}$  &&$79$ & $0$--$\mathbf{7}$  \\
$3^5$ & $2$--$\mathbf{21}$ &&$11^5$ & $0$--$\mathbf{82}$  &&$19^5$ & $0$--$\mathbf{105}$ &&$83$ & $3$--$\mathbf{18}$ \\
$5$   & $8$--$10$          &&$13$   & $6$--$\mathbf{11}$  &&$23$   & $3$--$\mathbf{11}$  &&$89$ & $0$--$\mathbf{8}$  \\
$5^2$ & $4$--$\mathbf{12}$ &&$13^2$ & $0$--$\mathbf{12}$  &&$29$   & $3$--$\mathbf{12}$  &&$97$ & $0$--$\mathbf{15}$ \\
$5^3$ & $3$--$\mathbf{20}$ &&$13^3$ & $0$--$\mathbf{37}$  &&$31$   & $6$--$\mathbf{17}$  \\
$5^4$ & $0$--$\mathbf{20}$ &&$13^4$ & $0$--$\mathbf{28}$  &&$37$   & $6$--$\mathbf{14}$  \\
$5^5$ & $0$--$\mathbf{39}$ &&$13^5$ & $0$--$\mathbf{60}$  &&$41$   & $2$--$\mathbf{10}$  \\
$7$   & $7$--$7$           &&$17$   & $4$--$\mathbf{14}$  &&$43$   & $7$--$\mathbf{11}$  \\
$7^2$ & $0$--$0$           &&$17^2$ & $0$--$0$            &&$47$   & $3$--$\mathbf{15}$  \\
$7^3$ & $3$--$\mathbf{27}$ &&$17^3$ & $2$--$\mathbf{26}$  &&$53$   & $4$--$\mathbf{12}$  \\
$7^4$ & $0$--$\mathbf{16}$ &&$17^4$ & $0$--$\mathbf{12}$  &&$59$   & $5$--$\mathbf{17}$  \\
$7^5$ & $4$--$\mathbf{45}$ &&$17^5$ & $0$--$\mathbf{123}$ &&$61$   & $3$--$\mathbf{11}$  \\
\bottomrule
\end{tabular}
\end{center}
\bigskip
\caption{The best upper and lower bounds known for the minimal defect $D_q(7)$ 
for the odd values of $q$ represented in the \manypointslink\ table, as of 
29 March 2017. The upper bounds in boldface come from computations described
in this paper.}
\label{table:7}
\end{table}

\section{Conclusion}

Even as we were writing this paper and implementing the algorithms described
here, other researchers were adding new lower bounds for values of $N_q(g)$ in
the \manypointslink\ table, including many for $g = 5$.  Some of these bounds
were better than the results obtained by our methods, some were not.  We hope
that the inclusion of our new lower bounds in the online table will encourage
further development of algorithms to produce curves with many points.

\bibliographystyle{hplaindoi} 
\bibliography{Howe-FFA}

\newcommand{\SortNoop}[1]{}
\providecommand{\bysame}{\leavevmode\hbox to3em{\hrulefill}\thinspace}
\providecommand{\xxMR}[2]{\relax\ifhmode\unskip\space\fi
  \href{http://www.ams.org/mathscinet-getitem?mr=#2}{MR~#1}}
\providecommand{\xxZBL}[1]{\relax\ifhmode\unskip\space\fi
  \href{http://www.emis.de/cgi-bin/MATH-item?#1}{ZBL~#1}}
\providecommand{\xxJFM}[1]{\relax\ifhmode\unskip\space\fi
  \href{http://www.emis.de/cgi-bin/JFM-item?#1}{JFM~#1}}
\providecommand{\xxARXIV}[2]{\relax\ifhmode\unskip\space\fi
  \href{http://arxiv.org/abs/#2}{arXiv:#1}}
\providecommand\bibmarginpar{\leavevmode\marginpar}
\providecommand{\href}[2]{#2}
\begin{thebibliography}{10}

\bibitem{BosmaCannonEtAl1993}
Wieb Bosma, John Cannon, and Catherine Playoust,
  \href{http://dx.doi.org/10.1006/jsco.1996.0125} {\emph{The {M}agma algebra
  system. {I}. {T}he user language}}, J. Symbolic Comput. \textbf{24} (1997),
  no.~3-4, 235--265, Computational algebra and number theory (London, 1993).
  Software available at \url{http://magma.maths.usyd.edu.au/}.

\bibitem{BrokerHoweEtAl2015}
Reinier Br{\"o}ker, Everett~W. Howe, Kristin~E. Lauter, and Peter Stevenhagen,
  \href{http://dx.doi.org/10.1112/S1461157014000461} {\emph{Genus-2 curves and
  {J}acobians with a given number of points}}, LMS J. Comput. Math. \textbf{18}
  (2015), no.~1, 170--197.

\bibitem{Deuring1941}
Max Deuring, \href{http://dx.doi.org/10.1007/BF02940746} {\emph{Die {T}ypen der
  {M}ultiplikatorenringe elliptischer {F}unktionenk{\"o}rper}}, Abh. Math. Sem.
  Hansischen Univ. \textbf{14} (1941), 197--272.

\bibitem{manypoints}
Gerard {\SortNoop{Geer}}van~der Geer, Everett~W. Howe, Kristin~E. Lauter, and
  Christophe Ritzenthaler, \href{http://www.manypoints.org}
  {\emph{\texttt{manypoints.org} --- {T}able of curves with many points}},
  2009, retrieved 29 March 2017.

\bibitem{GeerVlugt2000}
Gerard {\SortNoop{Geer}}van~der Geer and Marcel van~der Vlugt,
  \href{http://dx.doi.org/10.1090/S0025-5718-99-01143-6} {\emph{Tables of
  curves with many points}}, Math. Comp. \textbf{69} (2000), no.~230, 797--810.

\bibitem{HoweLauter2003}
E.~W. Howe and K.~E. Lauter,
  \href{http://aif.cedram.org/item?id=AIF_2003__53_6_1677_0} {\emph{Improved
  upper bounds for the number of points on curves over finite fields}}, Ann.
  Inst. Fourier (Grenoble) \textbf{53} (2003), no.~6, 1677--1737, Corrigendum:
  \textbf{57} (2007) 1019--1021.

\bibitem{Howe2012}
Everett~W. Howe, \href{http://bookstore.ams.org/conm-574/} {\emph{New bounds on
  the maximum number of points on genus-4 curves over small finite fields}},
  Arithmetic, Geometry, Cryptography and Coding Theory (Y.~Aubry,
  C.~Ritzenthaler, and A.~Zykin, eds.), Contemp. Math., vol. 574, American
  Mathematical Society, Providence, RI, 2012, pp.~69--86.

\bibitem{Howe2016}
\bysame, \href{http://bookstore.ams.org/conm-663/} {\emph{Quickly constructing
  curves of genus 4 with many points}}, Frobenius Distributions: {S}ato--{T}ate
  and {L}ang--{T}rotter conjectures (D.~Kohel and I.~Shparlinski, eds.),
  Contemp. Math., vol. 663, American Mathematical Society, Providence, RI,
  2016, pp.~149--173.

\bibitem{HoweLauter2012}
Everett~W. Howe and Kristin~E. Lauter,
  \href{http://dx.doi.org/10.4171/119-1/12} {\emph{New methods for bounding the
  number of points on curves over finite fields}}, Geometry and arithmetic
  (C.~Faber, G.~Farkas, and R.~de~Jong, eds.), EMS Ser. Congr. Rep., European
  Mathematical Society, Z\"urich, 2012, pp.~173--212.

\bibitem{HoweLeprevostEtAl2000}
Everett~W. Howe, Franck Lepr{\'e}vost, and Bjorn Poonen,
  \href{http://dx.doi.org/10.1515/form.2000.008} {\emph{Large torsion subgroups
  of split {J}acobians of curves of genus two or three}}, Forum Math.
  \textbf{12} (2000), no.~3, 315--364.

\bibitem{HoweNartEtAl2009}
Everett~W. Howe, Enric Nart, and Christophe Ritzenthaler,
  \href{http://dx.doi.org/10.5802/aif.2430} {\emph{Jacobians in isogeny classes
  of abelian surfaces over finite fields}}, Ann. Inst. Fourier (Grenoble)
  \textbf{59} (2009), no.~1, 239--289.

\bibitem{KaniRosen1989}
E.~Kani and M.~Rosen, \href{http://dx.doi.org/10.1007/BF01442878}
  {\emph{Idempotent relations and factors of {J}acobians}}, Math. Ann.
  \textbf{284} (1989), no.~2, 307--327.

\bibitem{Macbeath1965}
A.~M. Macbeath, \href{http://dx.doi.org/10.1112/plms/s3-15.1.527} {\emph{On a
  curve of genus {$7$}}}, Proc. London Math. Soc. (3) \textbf{15} (1965),
  527--542.

\bibitem{Mestre2010}
Jean-Fran{\c{c}}ois Mestre, \emph{Courbes de genre $3$ avec {$S_3$} comme
  groupe d'automorphismes}, 2010. \xxARXIV{1002.4751 [math.AG]}{1002.4751}

\bibitem{Serre1983b}
Jean-Pierre Serre,
  \href{http://resolver.sub.uni-goettingen.de/purl?GDZPPN002545039}
  {\emph{Nombres de points des courbes alg\'ebriques sur {${\bf {F}}_{q}$}}},
  Seminaire de Th\'eorie des Nombres de Bordeaux, 1982--1983, Univ. Bordeaux I,
  Talence, 1983, Exp. No. 22.

\bibitem{Serre1983a}
\bysame, \href{http://gallica.bnf.fr/ark:/12148/bpt6k31623/f592} {\emph{Sur le
  nombre des points rationnels d'une courbe alg\'ebrique sur un corps fini}},
  C. R. Acad. Sci. Paris S\'er. I Math. \textbf{296} (1983), no.~9, 397--402.

\bibitem{Serre1984}
\bysame, \emph{R\'esum\'e des cours de 1983--1984}, Ann. Coll\`ege France
  (1984), 79--83, $=$ \emph{{\OE}uvres} \textbf{132}.

\bibitem{Soomro2013}
Muhammad~Afzal Soomro,
  \href{http://www.rug.nl/research/portal/files/2371262/Afzal_Thesis.pdf}
  {\emph{Algebraic curves over finite fields}}, Ph.D. thesis, University of
  Groningen, 2013.

\bibitem{Waterhouse1969}
William~C. Waterhouse,
  \href{http://www.numdam.org/item?id=ASENS_1969_4_2_4_521_0} {\emph{Abelian
  varieties over finite fields}}, Ann. Sci. \'Ecole Norm. Sup. (4) \textbf{2}
  (1969), 521--560.

\end{thebibliography}

\end{document}